\documentclass{amsart}

\usepackage[latin1]{inputenc}
\usepackage[english]{babel}
\usepackage{indentfirst}
\usepackage{amssymb}
\usepackage{amsthm}
\usepackage{xcolor}
\usepackage[all]{xy}
\usepackage[mathscr]{eucal}
\usepackage{hyperref}

\newcommand{\impli}{\Rightarrow}
\newcommand{\Nat}{\mathbb{N}}
\newcommand{\N}{\mathbb{N}}

\newcommand{\sub}{\subseteq}

\def\epsilon{\varepsilon}

\newtheorem{theo}{Theorem}[section]
\newtheorem{lem}[theo]{Lemma}
\newtheorem{pro}[theo]{Proposition}
\newtheorem{cor}[theo]{Corollary}
\newtheorem{defi}[theo]{Definition}

\newtheorem{question}[theo]{Question}

\theoremstyle{definition}

\newtheorem{rem}[theo]{Remark}

\newcommand{\pten}{\ensuremath{\widehat{\otimes}_\pi}}

\numberwithin{equation}{section}

\title{Weak precompactness in projective tensor products}

\author[J. Rodr\'iguez]{Jos\'e Rodr\'{i}guez}
\address{Departamento de Ingenier\'{i}a y Tecnolog\'{i}a de Computadores,
Facultad de Inform\'{a}tica, Universidad de Murcia, 30100 Espinardo (Murcia), Spain.}  
\email{joserr@um.es}
\urladdr{\url{https://webs.um.es/joserr}}

\author[A. Rueda Zoca]{Abraham Rueda Zoca}
\address{Departamento de An\'{a}lisis Matem\'{a}tico, Facultad de Ciencias, Universidad de Granada, 18071 Granada, Spain.}
\email{abrahamrueda@ugr.es}
\urladdr{\url{https://arzenglish.wordpress.com}}



\thanks{The research was supported by grants PID2021-122126NB-C32 (J. Rodr\'{i}guez) and 
PID2021-122126NB-C31 (A. Rueda Zoca), funded by MCIN/AEI/10.13039/501100011033 and ``ERDF A way of making Europe'', and also  
by grant 21955/PI/22 (funded by Fundaci\'on S\'eneca - ACyT Regi\'{o}n de Murcia).
The research of A. Rueda Zoca was also supported by grants FQM-0185 and PY20\_00255 (funded by Junta de Andaluc\'ia).}

\begin{document}

\begin{abstract}
We give a sufficient condition for a pair of Banach spaces $(X,Y)$ to have the following property: 
whenever $W_1 \sub X$ and $W_2 \sub Y$ are sets such that $\{x\otimes y: \, x\in W_1, \, y\in W_2\}$ is weakly precompact in the projective tensor product
$X\pten Y$, then either $W_1$ or $W_2$ is relatively norm compact. For instance, such a property holds
for the pair $(\ell_p,\ell_q)$ if $1<p,q<\infty$ satisfy $1/p+1/q\geq 1$. Other examples are given that 
allow us to provide alternative proofs to some results on multiplication operators due to Saksman and Tylli.
We also revisit, with more direct proofs, some known results about the embeddability of $\ell_1$ into $X\pten Y$ for arbitrary Banach spaces $X$ and~$Y$, in connection
with the compactness of all operators from $X$ to~$Y^*$.
\end{abstract}

\maketitle

\section{Introduction}

Let $\mathcal{L}(X)$ denote the Banach space of all (bounded linear) operators on a Banach space~$X$. Given
$R,S\in \mathcal{L}(X)$ one can consider the multiplication operator 
$$
	\Phi_{R,S}:\mathcal{L}(X)\to \mathcal{L}(X)
$$ 
defined by 
$$
	\Phi_{R,S}(T):=R\circ T \circ S
	\quad\text{for all $T\in \mathcal{L}(X)$}.
$$
Operator ideal properties of such multiplication operators have been widely studied in the literature.
As to weak compactness, it is known that $\Phi_{R,S}$ is weakly compact whenever $R$ is compact and $S$ is weakly compact, or vice versa
(see \cite[Theorem~2.9]{sak-tyl-1}). In the other direction, it is not difficult to check that both $R$ and $S$
are weakly compact whenever they are non-zero and $\Phi_{R,S}$ is weakly compact, but in some cases one can say more.
Akemann and Wright (see \cite[Proposition~2.3]{ake-wri}) proved that for $X=\ell_2$ the weak compactness of $\Phi_{R,S}$ implies
that either $R$ or $S$ is compact. Later, Saksman and Tylli (see~\cite[Propositions~3.2 and~3.8]{sak-tyl-1}) showed 
that this property holds when $X$ is a subspace of~$\ell_p$ for $1<p<\infty$ or 
$X$ is the James space. In general, this is not true for arbitrary Banach spaces. See \cite{joh-sch,lin-sch,rac,sak-tyl-1,sak-tyl-2,sak-tyl-3} for more information on this topic.  

The previous circle of ideas is intimately related to weak compactness in projective tensor products. Let $X$ and $Y$ be Banach spaces, let
$X\pten Y$ be its projective tensor product and let $W_1 \sub X$ and $W_2 \sub Y$. Then the set
$$
	W_1 \otimes W_2:=\{x\otimes y: \, x\in W_1, \, y\in W_2\} \sub X\pten Y
$$ 
is relatively weakly compact whenever $W_1$ is relatively norm compact and $W_2$ is relatively weakly compact, or vice versa. In general, relative weak compactness
of both $W_1$ and $W_2$ is not sufficient for $W_1\otimes W_2$ to be relatively weakly compact, neither weakly precompact. 
Recall that a subset of a Banach space is said to be {\em weakly precompact} (or {\em conditionally weakly compact})
if every sequence in it admits a weakly Cauchy subsequence or, equivalently (thanks to Rosenthal's $\ell_1$-theorem; see, e.g., \cite[Theorem~10.2.1]{alb-kal}), 
if the set is bounded and contains no {\em $\ell_1$-sequence} (that is, a basic sequence which is equivalent to the usual basis of~$\ell_1$).
For instance, if $1<p,q<\infty$ satisfy $1/p+1/q\geq 1$, then the sequence $(e_n\otimes e'_n)_{n\in \N}$ in $\ell_p \pten \ell_q$ is an $\ell_1$-sequence, where
we denote by $(e_n)_{n\in \N}$ and $(e'_n)_{n\in \N}$ the usual bases of $\ell_p$ and~$\ell_q$, respectively
(see, e.g., the proof of \cite[Proposition~3.6]{avi-alt-8}). The following definition arises naturally:  

\begin{defi}\label{defi:bird}
Let $X$ and $Y$ be Banach spaces. We say that the pair $(X,Y)$ has {\em property (AW)}
if whenever $W_1 \sub X$ and $W_2 \sub Y$ are sets such that $W_1\otimes W_2$ is weakly precompact in $X\pten Y$, then
either $W_1$ or $W_2$ is relatively norm compact.
\end{defi}

Clearly, if $X$ and $Y$ are infinite-dimensional Banach spaces such that $X\pten Y$ contains no subspace isomorphic to~$\ell_1$, then
the pair $(X,Y)$ fails property~(AW) (the unit balls $B_X$ and $B_Y$ fail to be norm compact, while $B_X\otimes B_Y$ is weakly precompact in $X\pten Y$). 
Such an example is given by $(\ell_p,\ell_q)$ for $1<p,q<\infty$ with $1/p+1/q<1$, because
in this case $\ell_p\pten \ell_q$ is reflexive (see, e.g., \cite[Corollary~4.24]{rya}). In \cite[Proposition~3.17]{rod21} it is shown
that the pair $(X,Y)$ has property~(AW) whenever $X$ and $Y$ are Banach spaces 
with unconditional finite-dimensional Schauder decompositions having a disjoint lower $p$-estimate and a disjoint lower $q$-estimate, respectively, 
where $1< p,q<\infty$ satisfy $1/p+1/q\geq 1$. In particular, for $1<p,q<\infty$, the pair $(\ell_p,\ell_q)$ has property~(AW)
if and only if $1/p+1/q\geq 1$.

The aim of this paper is to go a bit further in the analysis of $\ell_1$-sequences in projective tensor products of Banach spaces, property (AW) and 
its applications to multiplication operators. The paper is organized as follows.

In Section~\ref{section:multiplication} we discuss the relationship between multiplication operators and projective tensor products, specially
in connection with weak compactness. Some known results are included for the sake of completeness.

In Section~\ref{section:AW} we focus on property (AW). The following property plays an important role in our discussion:

\begin{defi}\label{defi:Rp-intro}
Let $X$ be a Banach space and $1< p< \infty$. We say that $X$ has {\em property~$(R_p)$} if
for every relatively weakly compact set $A \sub X$ which is not relatively norm compact there is an operator
$u: X\to \ell_p$ such that $u(A)$ is not relatively norm compact.
\end{defi}

Obviously, $\ell_p$ has property~$(R_p)$ for every $1<p<\infty$. This property is closely related to the class of coarse $p$-limited sets 
introduced in~\cite{gal-mir} and, in particular, it agrees with the so-called coarse $p$-Gelfand-Phillips property when $2\leq p <\infty$
(see Remark~\ref{rem:p-limited-2}). We prove that the pair $(X,Y)$ has property~(AW)
whenever $X$ and $Y$ are Banach spaces having properties $(R_p)$ and~$(R_q)$, respectively,
where $1< p,q<\infty$ satisfy $1/p+1/q\geq 1$ (see Theorem~\ref{theo:Rp}). One of the possible approaches to the previous result 
sheds some more light on $\ell_1$-sequences in this setting: under the same assumptions on~$X$ and~$Y$, 
if $(x_n)_{n\in \N}$ and $(y_n)_{n\in \N}$ are weakly Cauchy sequences in~$X$ and~$Y$, respectively, without
norm convergent subsequences, then $(x_n\otimes y_n)_{n\in \N}$ admits an $\ell_1$-subsequence in~$X\pten Y$
(see Theorem~\ref{theo:Rp0}). As an application of Theorem~\ref{theo:Rp} and some results of Knaust and Odell~\cite{kna-ode}, we provide new proofs
of the aforementioned results on multiplication operators of Saksman and Tylli (see Corollaries~\ref{cor:subspace-lp-operator} and~\ref{cor:James-operator}).

In Section~\ref{section:l1} we include some complements about $\ell_1$-sequences in projective tensor products
and we provide more direct proofs of some known results about the embeddability of $\ell_1$ into $X\pten Y$ for arbitrary Banach spaces $X$ and~$Y$, due
to Emmanuele~\cite{emm2} and Xue, Li and Bu~\cite{bu-alt}. Namely: 
\begin{enumerate}
\item[(i)] If $X$ and $Y$ contain no subspace isomorphic to~$\ell_1$ and all operators from $X$ to~$Y^*$ are compact, then 
$X\pten Y$ contains no subspace isomorphic to~$\ell_1$, \cite[Theorem~3]{emm2} (see Theorem~\ref{theo:Emmanuele}).
\item[(ii)] If $X\pten Y$ contains no subspace isomorphic to~$\ell_1$ and either $X$ or $Y$
has an unconditional basis, then all operators from $X$ to~$Y^*$ are compact, \cite[Corollary~6]{emm2} and \cite[Theorem~4]{bu-alt} 
(see Theorem~\ref{theo:Bu}).
\end{enumerate}

\subsection*{Terminology} 
We work with real Banach spaces.
Let $X$ be a Banach space. The norm of~$X$ is denoted by $\|\cdot\|_{X}$ or simply $\|\cdot\|$. The topological dual of~$X$ is denoted by~$X^*$.
The evaluation of $x^*\in X^*$ at $x\in X$ is denoted by either $x^*(x)$ or $\langle x^*,x\rangle$. 
By a {\em subspace} of~$X$ we mean a norm closed linear subspace.
Given a set $C \sub X$, its closed convex hull and its closed linear span (i.e., the subspace of~$X$ generated by~$C$) are denoted by
$\overline{{\rm conv}}(C)$ and $\overline{{\rm span}}(C)$, respectively.
The closed unit ball of~$X$ is $B_X=\{x\in X:\|x\|\leq 1\}$. 
Given two sets $C_1,C_2 \sub X$, its Minkowski sum is $C_1+C_2:=\{x_1+x_2:\, x_1\in C_1, \, x_2\in C_2\}$.
By an {\em operator} we mean a continuous linear map between Banach spaces. Given another Banach space~$Y$, we denote by
$\mathcal{L}(X,Y)$ the Banach space of all operators from~$X$ to~$Y$, equipped with the operator norm (when $X=Y$ we just write $\mathcal{L}(X)$
instead). As usual, we denote by $T^*\in \mathcal{L}(Y^*,X^*)$ the adjoint of $T\in \mathcal{L}(X,Y)$.

We denote by $\mathcal{B}(X,Y)$
the Banach space of all continuous bilinear functionals on $X\times Y$, with the norm
$\|S\|_{\mathcal{B}(X,Y)}=\sup\{|S(x,y)|:\, x \in B_X,\, y \in B_Y\}$. Observe that the spaces
$\mathcal{B}(X,Y)$, $\mathcal{L}(X,Y^*)$ and $\mathcal{L}(Y,X^*)$ are isometric in the natural way.
Each $S\in \mathcal{B}(X,Y)$ induces
a linear functional $\tilde{S}$ in the algebraic tensor product $X\otimes Y$.
The {\em projective tensor product} of~$X$ and~$Y$, denoted by $X\pten Y$, is the completion
of $X\otimes Y$ when equipped with the norm
$$
	\|z\|=\sup\{|\tilde{S}(z)|: \, S\in \mathcal{B}(X,Y), \, \|S\|\leq 1\}, \quad z \in X\otimes Y.
$$
Hence, each $S\in \mathcal{B}(X,Y)$ induces an element of $(X\pten Y)^*$ (namely, the continuous linear extension of~$\tilde{S}$ to $X\pten Y$). In fact,
this gives an onto isometry between $\mathcal{B}(X,Y)$ and~$(X \pten Y)^*$ (see, e.g., \cite[Section~2.2]{rya}). In the sequel we will identify the spaces
$(X \pten Y)^*$, $\mathcal{B}(X,Y)$, $\mathcal{L}(X,Y^*)$ and $\mathcal{L}(Y,X^*)$ via that isometry.

\section{Multiplication operators and tensor products}\label{section:multiplication}

In this preliminary section we discuss the relationship between multiplication operators and tensor products of operators, in connection with weak compactness. 
While most of the results are already known, we include their proofs, which can help readers to focus on the subject.

\begin{lem}\label{lem:compact-tensor-rwc}
Let $X_1$ and $X_2$ be Banach spaces and let $C_1 \sub X_1$ and $C_2 \sub X_2$. The following statements hold:
\begin{enumerate}
\item[(i)] If $C_1 \otimes C_2$ is relatively weakly compact in $X_1\pten X_2$, then both $C_1$ and $C_2$ are relatively weakly compact
provided that they are not equal to~$\{0\}$. The same holds if relative weak compactness is replaced by weak precompactness.
\item[(ii)] If $C_1$ is relatively norm compact and $C_2$ is relatively weakly compact (resp., weakly precompact), 
then $C_1 \otimes C_2$ is relatively weakly compact (resp., weakly precompact) in $X_1\pten X_2$.
\end{enumerate}
\end{lem}
\begin{proof}
(i) Fix $x_i\in C_i\setminus \{0\}$ for $i\in \{1,2\}$ and consider the isomorphic embeddings 
$$
	\iota_1:X_1 \to X_1\pten X_2 
	\quad\mbox{and} 
	\quad 
	\iota_2:X_2 \to X_1\pten X_2
$$
given by $\iota_1(x):=x \otimes x_2$ for all $x\in X_1$ and $\iota_2(y):=x_1\otimes y$ for all $y\in X_2$. 
Since both $\iota_1(C_1)$ and $\iota_2(C_2)$ are contained in $C_1\otimes C_2$, the conclusion follows at once.

(ii) Suppose that $C_2$ is relatively weakly compact. We can assume without loss of generality that $C_1 \sub B_{X_1}$ and $C_2\sub B_{X_2}$. 
Let $(x_n)_{n\in \N}$ and $(y_n)_{n\in \N}$ be sequences in~$C_1$ and~$C_2$, respectively. By passing to subsequences,
we can assume that $(x_n)_{n\in \N}$ is norm convergent to some $x\in X_1$ and that
$(y_n)_{n\in \N}$ is weakly convergent to some $y\in X_2$. Given any $T\in \mathcal{L}(X,Y^*)$, for each $n\in \N$ we have
\begin{eqnarray*}
	\big|\langle T,x_n\otimes y_n\rangle - \langle T,x\otimes y\rangle\big| & \leq &
	\big|\langle T(x_n)-T(x), y_n\rangle \big| + \big|\langle T(x),y_n-y\rangle\big|  \\
	& \leq & \|T\| \|x_n-x\| + \big|\langle T(x),y_n-y\rangle\big|
\end{eqnarray*}  
and so $\langle T,x_n\otimes y_n\rangle \to \langle T,x\otimes y\rangle$ as $n\to \infty$.
This shows that $(x_n\otimes y_n)_{n\in \N}$ is weakly convergent to $x\otimes y$ in $X_1\pten X_2$.
The proof that $C_1\otimes C_2$ is weakly precompact when~$C_2$ is weakly precompact is similar.
\end{proof}

\begin{rem}\label{rem:DP}
In the setting of Lemma~\ref{lem:compact-tensor-rwc}, a similar argument shows that, if either $X_1$ or $X_2$ has the Dunford-Pettis and 
both $C_1$ and $C_2$ are relatively weakly compact, then $C_1 \otimes C_2$ is relatively weakly compact in $X_1\pten X_2$
(which reproves a result of J.~Diestel, see \cite[Theorem~16]{die8}).
\end{rem}

Given two operators $T_1:Y_1\to X_1$ and $T_2:Y_2 \to X_2$, where $Y_1$, $Y_2$, $X_1$ and~$X_2$ are Banach spaces, the {\em projective tensor product}
of $T_1$ and $T_2$ is the unique operator 
$$
	T_1\otimes T_2: Y_1\pten Y_2 \to X_1 \pten X_2
$$ 
satisfying 
$$
	(T_1\otimes T_2)(y_1\otimes y_2)=T_1(y_1)\otimes T_2(y_2)
$$
for every $y_1\in Y_1$ and for every $y_2\in Y_2$ (see \cite[Proposition 2.3]{rya} for details).

\begin{lem}\label{lem:tensor-operator-rwc}
Let $Y_1$, $Y_2$, $X_1$ and~$X_2$ be Banach spaces and let $T_1:Y_1\to X_1$ and $T_2:Y_2\to X_2$ be operators. Then 
$T_1\otimes T_2$ is weakly compact if and only if $T_1(B_{Y_1})\otimes T_2(B_{Y_2})$ is relatively weakly compact in $X_1\pten X_2$. 
\end{lem}
\begin{proof}
We have $B_{Y_1\pten Y_2}=\overline{{\rm conv}}(B_{Y_1}\otimes B_{Y_2})$ (see, e.g., \cite[Proposition~2.2]{rya})
and therefore the set $W:=T_1(B_{Y_1})\otimes T_2(B_{Y_2})=(T_1\otimes T_2)(B_{Y_1}\otimes B_{Y_2})$ satisfies
$$
	\overline{{\rm conv}}(W)=\overline{(T_1\otimes T_2)(B_{Y_1\pten Y_2})}.
$$
The conclusion now follows from the Krein-\v Smulyan theorem asserting that the convex hull of a relatively weakly compact subset
of an arbitrary Banach space is relatively weakly compact as well (see, e.g., \cite[p.~51, Theorem~11]{die-uhl-J}). 
\end{proof}

\begin{pro}\label{pro:wedge-vs-tensor}
Let $X$, $X_1$, $Y$ and $Y_1$ be Banach spaces and let $S:X_1 \to X$ and $R:Y\to Y_1$ be operators. Let us consider the operator
$$
	\Phi_{R^{**},S}: \mathcal{L}(X,Y^{**}) \to \mathcal{L}(X_1,Y_1^{**})
$$
defined by
$$
	\Phi_{R^{**},S}(T):=R^{**}\circ T \circ S
	\quad
	\text{for all $T\in \mathcal{L}(X,Y^{**})$}.
$$
Then: 
\begin{enumerate}
\item[(i)] $\Phi_{R^{**},S}=(S\otimes R^{*})^{*}$, where as usual we identify 
$\mathcal{L}(X,Y^{**})=(X\pten Y^*)^*$ and $\mathcal{L}(X_1,Y_1^{**})=(X_1\pten Y_1^*)^*$.

\item[(ii)] $\Phi_{R^{**},S}$ is weakly compact if and only if $S\otimes R^*$ is weakly compact. In this case, 
the operator
$$
	\Phi_{R,S}: \mathcal{L}(X,Y) \to \mathcal{L}(X_1,Y_1)
$$
defined by
$$
	\Phi_{R,S}(T):=R\circ T \circ S \quad \text{for all $T\in \mathcal{L}(X,Y)$}
$$
is weakly compact and if, in addition, both $S$ and $R$ are non-zero, then they are weakly compact as well.
\end{enumerate}
\end{pro}
\begin{proof}
(i) Fix $T\in \mathcal{L}(X,Y^{**})$. Then $(S\otimes R^{*})^{*}(T)=T \circ (S \otimes R^*)$.
Given arbitrary $x\in X_1$ and $y^*\in Y_1^*$, we have
\begin{multline*}
	\big\langle (S\otimes R^{*})^{*}(T),x\otimes y^*\big\rangle =\big(T \circ (S \otimes R^*)\big)(x\otimes y^*)=
	T\big(S(x)\otimes R^*(y^*)\big) \\ =
	\big\langle T(S(x)),R^*(y^*)\big\rangle=
	\big\langle R^{**}(T(S(x))),y^*\big\rangle=
	\langle \Phi_{R^{**},S}(T),x\otimes y^*\rangle.
\end{multline*}
Hence, $(S\otimes R^{*})^{*}(T)=\Phi_{R^{**},S}(T)$.

(ii) The first statement follows from Gantmacher's theorem and~(i). Note that
the weak compactness of $S\otimes R^*$ is equivalent to the relative weak compactness
of the set $S(B_{X_1})\otimes R^*(B_{Y_1^*})$ in $X_1\pten Y_1^*$
(see Lemma~\ref{lem:tensor-operator-rwc}). Therefore, the second statement is a consequence of Lemma~\ref{lem:compact-tensor-rwc}(i)
and the fact that $\Phi_{R,S}$ is the restriction of $\Phi_{R^{**},S}$ to $\mathcal{L}(X,Y)$
as a subspace of $\mathcal{L}(X,Y^{**})$.
\end{proof}

The following result (in slightly less generality) was first proved in \cite[Theorem~2.9]{sak-tyl-1}. See \cite{lin-sch,rac,sak-tyl-2,sak-tyl-3} for other proofs. Our approach
is close to that of \cite[Proposition~1]{rac} and \cite[Proposition~2.3(ii)]{sak-tyl-3}.

\begin{cor}\label{cor:compact-tensor-weaklycompact}
Let $S$ and $R$ be as in Proposition~\ref{pro:wedge-vs-tensor}.
Suppose that either (i)~$S$ is compact and $R$ is weakly compact or (ii)~$S$ is weakly compact and $R$ is compact.
Then $\Phi_{R^{**},S}$ is weakly compact. 
\end{cor}
\begin{proof}
We just prove case~(ii) as the other one is similar. Since $R$ is compact, Schauder's theorem ensures
that $R^*$ is compact too. Hence, $S(B_{X_1})$ is relatively weakly compact in~$X$ and $R^*(B_{Y_1^*})$ is relatively norm compact in~$Y^*$.
Then $S(B_{X_1})\otimes R^*(B_{Y_1^*})$ is relatively weakly compact in $X\pten Y^*$ (see Lemma~\ref{lem:compact-tensor-rwc}(ii))
and so Lemma~\ref{lem:tensor-operator-rwc} applies to deduce that $S\otimes R^*$ is a weakly compact operator.
The conclusion now follows from Proposition~\ref{pro:wedge-vs-tensor}(ii).
\end{proof}

Observe that the previous arguments and Remark~\ref{rem:DP} also lead to the next result going back to \cite[Proposition~2]{rac}:

\begin{rem}\label{rem:Racher}
Let $S$ and $R$ be as in Proposition~\ref{pro:wedge-vs-tensor}. Suppose that $X$ or $Y^*$ has the Dunford-Pettis property.
If both $S$ and $R$ are weakly compact, then $\Phi_{R^{**},S}$ is weakly compact.
\end{rem}

\subsection{An observation on the Davis-Figiel-Johnson-Pe{\l}czy\'nski factorization}\label{subsection:DFJP}
Let us recall the remarkable procedure that Davis, Figiel, Johnson and Pe{\l}czy\'nski invented in~\cite{dav-alt}. 
Let $X$ be a Banach space and let $W \sub X$ be an absolutely convex bounded set. For each $n\in \N$, 
denote by $|\cdot|_n$ the Minkowski functional of the absolutely convex bounded set~$W_n:=2^n W+2^{-n} B_X \sub X$, that is,
$$
	|x|_n:=\inf\{t>0: \, x\in tW_n\}
	\quad
	\text{for all $x\in X$.}
$$
Then $X_W:=\{x\in X: \ \sum_{n=1}^\infty |x|_n^2 <\infty\}$ is a linear subspace of~$X$ which becomes 
a Banach space when equipped with the norm 
$$
	\|x\|_{X_W}:=\left(\sum_{n=1}^\infty |x|_n^2\right)^{1/2}.
$$
The identity map $J_W: X_W \to X$ is an operator and $W \sub J_W(B_{X_W})$. Moreover, the space
$X_W$ is reflexive if and only if $W$ is relatively weakly compact. 
The operator $J_W$ will be called the {\em DFJP operator
associated to~$W$}. The reader can find in \cite[Section~5.2]{ali-bur} the basics on this topic.

The absolutely convex hull (resp., closed absolutely convex hull) of a subset $C$ of a Banach space will be denoted by ${\rm aconv}(C)$
(resp., $\overline{{\rm aconv}}(C)$).

\begin{pro}\label{pro:tensor-operator-DFJP}
Let $X_1$ and~$X_2$ be Banach spaces. For each $i\in \{1,2\}$, let $C_i \sub X_i$ be a bounded set and let $T_i: Y_i \to X_i$ be the DFJP operator associated 
to~$W_i:={\rm aconv}(C_i) \sub X_i$.
Then $T_1\otimes T_2$ is weakly compact if and only if $C_1\otimes C_2$ is relatively weakly compact in $X_1\pten X_2$.
\end{pro}
\begin{proof}
Since 
$$
	(T_1\otimes T_2)(B_{Y_1}\otimes B_{Y_2}) = T_1(B_{Y_1})\otimes T_2(B_{Y_2}) \supseteq W_1 \otimes W_2 \supseteq C_1\otimes C_2,
$$ 
the set $C_1\otimes C_2$ is relatively weakly compact in~$X_1\pten X_2$ whenever $T_1\otimes T_2$ is weakly compact.

To prove the converse, let us assume that $C_1\otimes C_2$ is relatively weakly compact
in $X_1\pten X_2$. By the Krein-\v Smulyan theorem (see, e.g., \cite[p.~51, Theorem~11]{die-uhl-J}), 
its absolutely convex hull ${\rm aconv}(C_1\otimes C_2)$ is relatively weakly compact in $X_1\pten X_2$ and, hence, the same holds 
for $W_1\otimes W_2 \sub {\rm aconv}(C_1\otimes C_2)$. The statement is obvious if either $W_1$ or $W_2$ equals to~$\{0\}$, so
we assume that this is not the case. Then $W_1$ and $W_2$ are relatively weakly compact (see Lemma~\ref{lem:compact-tensor-rwc}(i)). 

Fix $\epsilon>0$. Then there is $n\in \N$ such that 
$$
	T_i(B_{Y_i}) \sub n \overline{W_i}+\epsilon B_{X_i} 
	\quad \text{for $i\in \{1,2\}$}
$$
(see, e.g., the proof of \cite[Theorem~5.37(4)]{ali-bur}). Writing $\rho_i:=\|T_i\|^{-1}$, we have
$$
	H_i:=\rho_i T_i(B_{Y_i}) \sub \rho_i n \overline{W_i}+\rho_i\epsilon B_{X_i} 
	\quad \text{for $i\in \{1,2\}$}.
$$
Since $H_i \sub B_{X_i}$, the set 
$$
	U_i:=\frac{\rho_i n}{1+\rho_i \epsilon}\overline{W_i} \cap B_{X_i} 
$$
satisfies
$$
	H_i \sub U_i + \frac{2\rho_i\epsilon}{1+\rho_i\epsilon}B_{X_i}
	\quad\text{for $i\in \{1,2\}$}
$$
(see \cite[Lemma~3.10]{avi-ple-rod-5}). It follows that
\begin{equation}\label{equation:G}
	T_1(B_{Y_1}) \otimes T_2(B_{Y_2}) \sub V + f(\epsilon)B_{X_1\pten X_2},
\end{equation}
where
$$
	V:=\rho_1^{-1} U_1\otimes \rho_2^{-1} U_2 \sub X_1\pten X_2
$$
and
$$
	f(\epsilon):=2\epsilon\left(\frac{1}{1+\rho_1\epsilon}+\frac{1}{1+\rho_2\epsilon}+
	\left(\frac{1}{1+\rho_1\epsilon}\right)\left(\frac{2\epsilon}{1+\rho_2\epsilon}\right)\right).
$$
Writing $\theta_i:= n(1+\rho_i \epsilon)^{-1}$ for $i\in \{1,2\}$, we have
$$	
	V \sub \theta_1 \overline{W_1} \otimes \theta_2 \overline{W_2}=\theta_1\theta_2 \big(\overline{W_1}\otimes \overline{W_2}\big)
	\sub \theta_1\theta_2 \overline{W_1\otimes W_2} \sub \theta_1\theta_2 \overline{{\rm aconv}}(C_1\otimes C_2),
$$
and so \eqref{equation:G} yields
$$
		T_1(B_{Y_1}) \otimes T_2(B_{Y_2}) \sub \theta_1\theta_2 \overline{{\rm aconv}}(C_1\otimes C_2) + f(\epsilon)B_{X_1\pten X_2}.
$$
Notice that $\theta_1\theta_2 \overline{{\rm aconv}}(C_1\otimes C_2)$ is weakly compact in $X_1\pten X_2$ and that
$f(\epsilon)\to 0$ as $\epsilon\to 0$. It follows that $T_1(B_{Y_1}) \otimes T_2(B_{Y_2})$ is relatively weakly compact 
in $X_1\pten X_2$ (see, e.g., \cite[Lemma~13.32]{fab-ultimo}).
Finally, an appeal to Lemma~\ref{lem:tensor-operator-rwc} ensures that the operator $T_1\otimes T_2$ is weakly compact.
\end{proof}

\section{Property~(AW)}\label{section:AW}

The following notion was introduced in~\cite{gal-mir}:

\begin{defi}\label{defi:p-limited-set}
Let $X$ be a Banach space and $1< p< \infty$. A set $A \sub X$ is said to be {\em coarse $p$-limited} if 
$T(A)$ is relatively norm compact for every $T\in \mathcal{L}(X,\ell_p)$.
\end{defi}

\begin{rem}\label{rem:Rp-charact}
Let $X$ be a Banach space and $1< p< \infty$. The following statements are equivalent:
\begin{enumerate}
\item[(i)] $X$ has property $(R_p)$ (see Definition~\ref{defi:Rp-intro}), i.e., every coarse $p$-limited relatively weakly compact subset of~$X$ is relatively norm compact.
\item[(ii)] Every coarse $p$-limited weakly null sequence in~$X$ is norm null.
\item[(iii)] Every coarse $p$-limited weakly precompact subset of~$X$ is relatively norm compact.
\end{enumerate}
\end{rem}
\begin{proof}
The implications (iii)$\impli$(i)$\impli$(ii) are obvious. For (ii)$\impli$(iii), let $A \sub X$ be a coarse $p$-limited weakly precompact subset of~$X$. 
Let $(x_n)_{n\in \N}$ be a sequence in~$A$. By passing to a subsequence, 
we can assume that $(x_n)_{n\in \N}$ is weakly Cauchy. We claim that $(x_n)_{n\in \N}$ is norm Cauchy, which is enough to conclude (iii). Indeed, if this is not the case, then
we can find $\epsilon>0$ and two subsequences $(x_{n_k})_{k\in \N}$ and $(x_{m_k})_{k\in \N}$ such that $\|x_{n_{k}}-x_{m_k}\|>\epsilon$ for all $k\in \N$. Note that
for every $T\in \mathcal{L}(X,\ell_p)$ the set $T(A)$ is relatively norm compact in~$\ell_p$ and so the same holds for 
$$
	\{T(x_{n_{k}}-x_{m_k}): k\in \N\} \sub T(A)-T(A).
$$ 
Hence, $(x_{n_{k}}-x_{m_k})_{k\in \N}$ is a coarse $p$-limited, weakly null but not norm null sequence, a contradiction.
\end{proof}

\begin{rem}\label{rem:p-limited-2}
Let $X$ be a Banach space and $2\leq p <\infty$. Then every coarse $p$-limited subset of~$X$ is weakly precompact (see \cite[Proposition~3]{gal-mir}).
Consequently, $X$~has property~$(R_p)$ if and only if it has the {\em coarse $p$-Gelfand-Phillips property} of~\cite{gal-mir}, i.e.,
every coarse $p$-limited subset of~$X$ is relatively norm compact.
\end{rem}

The following result provides a sufficient condition on a pair of Banach spaces to have property~(AW) (see Definition~\ref{defi:bird}):

\begin{theo}\label{theo:Rp}
Let $X$ and $Y$ be Banach spaces such that $X$ has property~$(R_p)$ and  
$Y$ has property~$(R_q)$ for some $1< p,q<\infty$ satisfying $1/p+1/q\geq 1$. Then the pair $(X,Y)$ has property~(AW).
\end{theo}
\begin{proof}[First proof of Theorem~\ref{theo:Rp}]
By contradiction, suppose that there exist non relatively norm compact sets $W_1\sub X$ and $W_2 \sub Y$ such that
$W_1\otimes W_2$ is weakly precompact in $X\pten Y$. Then $W_1$ and $W_2$ are weakly precompact (see Lemma~\ref{lem:compact-tensor-rwc}(i)).
The assumptions on~$X$ and $Y$ ensure the existence of operators
$u:X\to \ell_p$ and $v:Y \to \ell_q$ such that $u(W_1)$ and $v(W_2)$ are not relatively norm compact. As we already mentioned in the introduction,
the condition $1/p+1/q\geq 1$ implies that the pair $(\ell_p,\ell_q)$ has property~(AW) and so
$u(W_1)\otimes u(W_2)$ is not weakly precompact in $\ell_p\pten \ell_q$. This is a contradiction, because
$u\otimes v: X\pten Y \to \ell_p\pten \ell_q$ is an operator, $W_1\otimes W_2$ is weakly precompact in $X\pten Y$
and $u(W_1)\otimes u(W_2)=(u\otimes v)(W_1\otimes W_2)$. 
\end{proof}

Theorem~\ref{theo:Rp0} below will provide a different approach to Theorem~\ref{theo:Rp}. It is convenient to
introduce first some terminology:

\begin{defi}\label{defi:Tp-plus}
Let $X$ be a Banach space and $1<p<\infty$. We say that $X$ has {\em property~$(P_p)$} if every seminormalized weakly null sequence $(x_n)_{n\in \N}$ in~$X$
admits a basic subsequence $(x_{n_k})_{k\in \N}$ which is equivalent to the usual basis of~$\ell_p$ 
and such that $\overline{{\rm span}}(\{x_{n_k}:k\in \N\})$ is complemented in~$X$.
\end{defi}

The following fact is well-known (see, e.g., \cite[Proposition~2.1.3]{alb-kal}): 

\begin{pro}\label{pro:lp}
For every $1< p<\infty$ the space $\ell_p$ has property~$(P_p)$.
\end{pro}

\begin{rem}\label{rem:Rp-Tp}
Let $X$ be a Banach space and $1< p< \infty$. If $X$ has property~$(P_p)$, then it also has property~$(R_p)$.
\end{rem}
\begin{proof}
It suffices to prove that any coarse $p$-limited weakly null sequence $(x_n)_{n\in \N}$ in~$X$ is norm null
(see Remark~\ref{rem:Rp-charact}).
By contradiction, suppose this is not the case. Then $(x_n)_{n\in \N}$ admits a seminormalized subsequence and so there is a further  
subsequence $(x_{n_k})_{k\in \N}$ which is a basic sequence equivalent to the usual basis $(e_k)_{k\in \N}$ of~$\ell_p$ and
such that $X_0:=\overline{{\rm span}}(\{x_{n_k}:k\in \N\})$ is complemented in~$X$. Let $T_0: X_0 \to \ell_p$ be the isomorphism satisfying $T(x_{n_k})=e_k$ for all $k\in \N$.  
Since $X_0$ is complemented in~$X$, we can extend $T_0$ to some operator $T:X \to\ell_p$. But $\{T(x_{n_k}):k\in \N\}=\{e_k:k\in \N\}$ is
not relatively norm compact in~$\ell_p$, which contradicts the fact that $(x_{n})_{n\in \N}$ is coarse $p$-limited.  
\end{proof}

\begin{theo}\label{theo:Rp0}
Let $X$ and $Y$ be Banach spaces such that $X$ has property~$(R_p)$ and  
$Y$ has property~$(R_q)$ for some $1< p,q<\infty$ satisfying $1/p+1/q\geq 1$.
Let $(x_n)_{n\in \N}$ and $(y_n)_{n\in \N}$ be weakly Cauchy sequences in~$X$ and~$Y$, respectively, without norm convergent subsequences.
Then $(x_n\otimes y_n)_{n\in \N}$ admits an $\ell_1$-subsequence in $X\pten Y$.
\end{theo}
\begin{proof}
The set $\{x_n:n\in \N\}$ is weakly precompact but fails to be norm relatively compact. 
Since $X$ has property~$(R_p)$, there is an operator $u:X \to \ell_p$ such that $\{u(x_n): n\in \N\}$
is not relatively norm compact in $\ell_p$. By passing to a subsequence, we can assume that
$(u(x_n))_{n\in \N}$ does not admit norm convergent subsequences.
Note that $(u(x_n))_{n\in \N}$ is weakly Cauchy and so weakly convergent to some $z\in\ell_p$ (the space $\ell_p$ 
is weakly sequentially complete). Define $z_n:=u(x_n)-z$ for all $n\in \N$.
Since the weakly null sequence $(z_n)_{n\in \N}$ does not admit norm null subsequences, by passing to a further subsequence
we can assume that $(z_{n})_{n\in \N}$ is a basic sequence equivalent to the usual basis of~$\ell_p$ and that
$\overline{{\rm span}}(\{z_n:n\in \N\})$ is complemented in~$\ell_p$ (see Proposition~\ref{pro:lp}). 
In the same way, since the set $\{y_{n}:n\in \N\}$ is weakly precompact but fails to be norm relatively compact,
property~$(R_q)$ of~$Y$ ensures the existence of an operator $v:Y \to \ell_q$,
a subsequence $(y_{n_k})_{k\in \N}$ and $w\in \ell_q$ such that the sequence $(w_k)_{k\in \N}$ 
defined by $w_k:=v(y_{n_k})-w$ for all $k\in \N$ is a basic sequence equivalent to the usual basis of~$\ell_q$ and that
$W:=\overline{{\rm span}}(\{w_k:k\in \N\})$ is complemented in~$\ell_q$.
 
Define $Z:=\overline{{\rm span}}(\{z_{n_k}:k\in \N\})$.
Then $(z_{n_k}\otimes w_k)_{k\in \N}$
is an $\ell_1$-sequence in~$Z\pten W$ (see, e.g., the proof of \cite[Proposition~3.6]{avi-alt-8}).
Let $\iota_Z: Z \to \ell_p$ and $\iota_W: W \to \ell_q$ be the inclusion operators.
Since $Z$ and $W$ are complemented in $\ell_p$ and~$\ell_q$, respectively,
the operator $\iota_Z\otimes \iota_W: Z \pten W \to \ell_p\pten \ell_q$ is an isomorphism onto a (complemented) subspace of~$\ell_p\pten \ell_q$
(see, e.g., \cite[Proposition~2.4]{rya}). Hence, $(z_{n_k}\otimes w_k)_{k\in \N}$
is also an $\ell_1$-sequence in~$\ell_p \pten \ell_q$. Since
$$
	h_k:=(u\otimes v)(x_{n_k}\otimes y_{n_k})=u(x_{n_k})\otimes v(y_{n_k})=z_{n_k}\otimes w_k + \underbrace{z \otimes w_k + z_{n_k}\otimes w + z\otimes w}_{=: h'_k}
$$
for all $k\in \N$ and the sequence $(h'_k)_{k\in \N}$ is weakly convergent (to~$z\otimes w$) in~$\ell_p\pten \ell_q$
(bear in mind that both $(z_{n_k})_{k\in \N}$ and $(w_k)_{k\in \N}$ are weakly null), 
we can apply Rosenthal's $\ell_1$-theorem (see, e.g., \cite[Theorem~10.2.1]{alb-kal}) to 
infer that $(h_k)_{k\in \N}$ admits an 
$\ell_1$-subsequence, say $(h_{k_j})_{j\in \N}$.
Since $u\otimes v$ is an operator, it is not difficult to prove that $(x_{n_{k_j}}\otimes y_{n_{k_j}})_{j\in \N}$
is an $\ell_1$-sequence in $X\pten Y$. This finishes the proof.
\end{proof}

Theorem \ref{theo:Rp0} provides an alternative proof of Theorem \ref{theo:Rp}, as follows.

\begin{proof}[Second proof of Theorem~\ref{theo:Rp}]
Let $W_1 \sub X$ and $W_2 \sub Y$ be sets such that $W_1\otimes W_2$ is weakly precompact in $X\pten Y$.
By contradiction, suppose that $W_1$ and $W_2$ are not relatively norm compact. Since both $W_1$ and $W_2$ 
are weakly precompact (see Lemma~\ref{lem:compact-tensor-rwc}(i)), there exist weakly Cauchy sequences 
$(x_n)_{n\in \N}$ in~$W_1$ and $(y_n)_{n\in \N}$ in~$W_2$ without norm convergent subsequences.  
By Theorem~\ref{theo:Rp0}, $(x_n\otimes y_n)_{n\in \N}$ admits an $\ell_1$-subsequence in $X\pten Y$, which contradicts
the weak precompactness of $W_1\otimes W_2$.
\end{proof}

Let us obtain another consequence of Theorem \ref{theo:Rp0} in the context of weakly null sequences in projective tensor products. 
Observe that the usual basis of $\ell_2$ shows that, in general, if $(x_n)_{n\in \N}$ and $(y_n)_{n\in \N}$ are weakly null sequences in~$X$ and~$Y$, respectively, 
the sequence $(x_n\otimes y_n)_{n\in \N}$ may fail to be weakly null in $X\pten Y$. To the best of our knowledge, the following question is open in complete generality. 

\begin{question}\label{question:Gonzalo}
Let $X$ and $Y$ be Banach spaces. Let $(x_n)_{n\in \N}$ and $(y_n)_{n\in \N}$ be weakly null sequences in~$X$ and~$Y$, respectively, such that
$(x_n\otimes y_n)_{n\in \N}$ is weakly convergent in~$X \pten Y$. Is $(x_n\otimes y_n)_{n\in \N}$ weakly null in~$X \pten Y$?
\end{question}

Theorem \ref{theo:Rp0} allows us to provide the following partial affirmative answer.

\begin{cor}\label{cor:tensor-weaklynull}
Let $X$ and $Y$ be Banach spaces such that $X$ has property~$(R_p)$ and  
$Y$ has property~$(R_q)$ for some $1< p,q<\infty$ satisfying $1/p+1/q\geq 1$.
Let $(x_n)_{n\in \N}$ and $(y_n)_{n\in \N}$ be weakly null sequences in~$X$ and~$Y$, respectively. Then 
$(x_n\otimes y_n)_{n\in \N}$ is weakly null if (and only if) it is weakly Cauchy in~$X \pten Y$.
\end{cor}
\begin{proof}
Suppose that $(x_n\otimes y_n)_{n\in \N}$ is weakly Cauchy. By Theorem~\ref{theo:Rp0}, either $(x_n)_{n\in \N}$ or $(y_n)_{n\in \N}$ admits a norm null subsequence. Hence,
$(x_n\otimes y_n)_{n\in \N}$ admits a weakly null subsequence (by the proof of Lemma~\ref{lem:compact-tensor-rwc}(ii)).
It follows that $(x_n\otimes y_n)_{n\in \N}$ is weakly null.
\end{proof}

\subsection{Some applications}\label{subsection:applications}

In this subsection we combine Theorem~\ref{theo:Rp} with a deep result of Knaust and Odell~\cite{kna-ode}
(see Theorem~\ref{theo:KO} below) to get some results on multiplication operators due to Saksman and Tylli~\cite{sak-tyl-1} 
(see Corollaries~\ref{cor:subspace-lp-operator} and~\ref{cor:James-operator}). We need to introduce some additional terminology. 

\begin{defi}\label{defi:Tp}
Let $X$ be a Banach space and $1< p< \infty$. A sequence $(x_n)_{n\in \N}$ in~$X$ is said to have an 
{\em upper $p$-estimate} if there is a constant $c>0$ such that
$$
	\left(\sum_{n=1}^m |a_n|^p\right)^{1/p} \geq c \left\|\sum_{n=1}^m a_n x_n\right\|
$$
for every $m\in \N$ and for all $a_1,\dots,a_m \in \mathbb{R}$.
\end{defi}

\begin{defi}\label{defi:Sp}
Let $X$ be a Banach space and $1< p< \infty$. We say $X$ has {\em property~$(S_p)$} if every seminormalized weakly null sequence in~$X$
admits a subsequence having an upper $p$-estimate.
\end{defi}

The following result can be found in \cite[Corollary~2]{kna-ode}:

\begin{theo}[Knaust-Odell]\label{theo:KO}
Let $X$ be a Banach space such that $X$ has property~$(S_p)$ and $X^*$ has property~$(S_{p'})$, where
$1<p,p'<\infty$ satisfy $1/p+1/p'=1$. The following statements hold:
\begin{enumerate}
\item[(i)] If $X^*$ contains no subspace isomorphic to~$\ell_1$, then $X$ has property~$(P_p)$.
\item[(ii)] If $X$ contains no subspace isomorphic to~$\ell_1$, then $X^*$ has property~$(P_{p'})$. 
\end{enumerate}
\end{theo}

\begin{cor}\label{cor:subspace-lp-KO}
Let $X$ be a subspace of~$\ell_p$, where $1<p<\infty$. Then:
\begin{enumerate} 
\item[(i)] $X$ has property~$(P_p)$.
\item[(ii)] $X^*$ has property~$(P_{p'})$, where $1<p'<\infty$ satisfies $1/p+1/p'=1$.
\item[(iii)] The pair $(X,X^*)$ has property~(AW).
\end{enumerate}
\end{cor}
\begin{proof}
Note that $\ell_{p}$ (resp., $\ell_{p'}$) has property~$(P_p)$ (resp., $(P_{p'})$), see
Proposition~\ref{pro:lp}, which in turn implies property~$(S_p)$ (resp., $(S_{p'})$). 
Clearly, property~$(S_p)$ is inherited by subspaces, so $X$ has property~$(S_p)$.
Since quotients of reflexive Banach spaces having property~$(S_{p'})$ also have property~$(S_{p'})$, it follows  
that $X^*$ has property~$(S_{p'})$. Statements~(i) and~(ii) now follow at once from Theorem~\ref{theo:KO}.
Statement~(iii) is a consequence of Theorem~\ref{theo:Rp} and Remark~\ref{rem:Rp-Tp}.
\end{proof}

\begin{cor}\label{cor:James-KO}
Let $J$ be the James space. Then:
\begin{enumerate}
\item[(i)] $J$ and $J^*$ have property $(P_2)$.
\item[(ii)] The pair $(J,J^*)$ has property~(AW).
\end{enumerate}
\end{cor}
\begin{proof}
It is known that both $J$ and $J^*$ have property~$(S_2)$, see \cite{ame-ito} and \cite[Proposition~3.3]{gon:95}, respectively.
Now, we can argue as in the proof of Corollary~\ref{cor:subspace-lp-KO}. 
\end{proof}

As an application we get the next result (see \cite[Proposition~3.2]{sak-tyl-1}):

\begin{cor}[Saksman-Tylli]\label{cor:subspace-lp-operator}
Let $X$ be a subspace of~$\ell_p$, where $1<p<\infty$, and let $R,S\in \mathcal{L}(X)$.
If the operator $\Phi_{R,S}: \mathcal{L}(X)\to \mathcal{L}(X)$ is weakly compact, then either $R$ or $S$ is compact.
\end{cor}
\begin{proof}
The weak compactness of~$\Phi_{R,S}$ is equivalent to the weak compactness of $S \otimes R^* \in \mathcal{L}(X\pten X^*)$ (see Proposition~\ref{pro:wedge-vs-tensor}), 
which in turn is equivalent to the relative weak compactness of $S(B_X)\otimes R^*(B_{X^*})$ in $X\pten X^*$
(see Lemma~\ref{lem:tensor-operator-rwc}). Since the pair $(X,X^*)$ has property~(AW) (see Corollary~\ref{cor:subspace-lp-KO}(iii)), 
either $S(B_X)$ or $R^*(B_{X^*})$ is relatively norm compact, that is, either $S$ or $R^*$ is a compact operator. In the second case,
Schauder's theorem applies to conclude that $R$ is compact as well.  
\end{proof}

Finally, the same argument applies to the following (see \cite[Proposition~3.8]{sak-tyl-1}):

\begin{cor}[Saksman-Tylli]\label{cor:James-operator}
Let $J$ be the James space and let $R,S\in \mathcal{L}(J)$.
If the operator $\Phi_{R,S}: \mathcal{L}(J)\to \mathcal{L}(J)$ is weakly compact, then either $R$ or $S$ is compact.
\end{cor}

\section{More on $\ell_1$-sequences in projective tensor products}\label{section:l1}

\subsection{Basic tensors  and unconditional bases}\label{subsection:repetitive}

If a sequence in a Banach space fails to be weakly null, then it admits a subsequence which is an {\em $\ell_1^+$-sequence}, in the following sense:

\begin{defi}\label{defi:l1plus}
Let $X$ be a Banach space. A sequence $(x_n)_{n\in \N}$ in~$X$
is called an {\em $\ell_1^+$-sequence} if the following equivalent statements hold:
\begin{enumerate}
\item[(i)] $0\not\in \overline{{\rm conv}}(\{x_n:n\in \N\})$.
\item[(ii)] There is $x^*\in X^*$ such that $x^*(x_n)\geq 1$ for every $n\in \Nat$.
\item[(iii)] There is a constant $C>0$ such that 
$$
	\left\|\sum_{n=1}^N a_n x_n\right\| \geq C \sum_{n=1}^N a_n
$$
for all $N\in \Nat$ and for all non-negative real numbers $a_1,\dots,a_N$.
\end{enumerate}
\end{defi}

It is not difficult to check that a bounded unconditional basic sequence is an $\ell_1$-sequence if and only if it is an $\ell_1^+$-sequence. 
In the same spirit, we have:  

\begin{lem}\label{lem:lplus}
Let $X$ and $Y$ be Banach spaces such that $X$ has an unconditional basis. Let~$(x_n)_{n\in \N}$ be a bounded unconditional basis of~$X$
and let $(y_n)_{n\in \N}$ be a bounded sequence in~$Y$. 
Then the sequence $(x_n\otimes y_n)_{n \in \N}$ in $X\pten Y$ is an $\ell_1$-sequence if and only if it is an $\ell_1^+$-sequence.
\end{lem}
\begin{proof}
Take an operator $T:X\to Y^*$ and $\epsilon>0$ such that $T(x_{n})(y_{n}) \geq \varepsilon$
for all $n\in \N$. Fix $N\in \N$ and $\lambda_1,\dots,\lambda_N \in \mathbb R$. Define an operator $G:X\to Y^*$ by 
$$
	G(x):=\sum_{n=1}^N {\rm sign}(\lambda_n) x_{n}^*(x) T(x_{n})
	\quad
 	\text{for all $x\in X$},
$$
where $(x_n^*)_{n\in \N}$ is the sequence in~$X^*$ of biorthogonal functionals associated to the basis $(x_n)_{n\in \N}$. 
Observe that for every $x \in X$ we have
\[
\begin{split}\Vert G(x)\Vert=\left\| \sum_{n=1}^N {\rm sign}(\lambda_n)x_{n}^*(x) T(x_{n})\right\|& =\left\| T\left(\sum_{n=1}^N {\rm sign}(\lambda_n)x_{n}^*(x) x_{n}\right)\right\|\\
& \leq \Vert T\Vert \left\| \sum_{n=1}^N {\rm sign}(\lambda_n)x_{n}^*(x) x_{n}\right\|\\
& \leq \Vert T\Vert K_u \Vert x\Vert,
\end{split}
\]
where $K_u$ stands for the unconditional basis constant of $(x_n)_{n\in \N}$. Hence,
we have $\|G\|\leq \|T\|K_u$ and therefore
\[
\begin{split}
\left\| \sum_{n=1}^N \lambda_n x_{n}\otimes y_{n}\right\|& \geq \frac{1}{\Vert T\Vert K_u}\sum_{n=1}^N \lambda_n G(x_{n})(y_{n})\\
& =\frac{1}{\Vert T\Vert K_u}\sum_{n=1}^N \lambda_n {\rm sign}(\lambda_n) T(x_{n})(y_{n})\\
& \geq \frac{\varepsilon}{\Vert T\Vert K_u}\sum_{n=1}^\infty \vert\lambda_n\vert.
\end{split}
\]
This shows that the (bounded) sequence $(x_{n}\otimes y_{n})_{n\in \N}$ is an $\ell_1$-sequence.
\end{proof}

The following result provides another partial affirmative answer to Question~\ref{question:Gonzalo}. The additional 
assumption on one of the Banach spaces is weaker than property~$(P_p)$ (see Definition~\ref{defi:Tp-plus})
and holds for $c_0$, all $\ell_p$ spaces with $1\leq p<\infty$ (see, e.g., \cite[Proposition~2.1.3]{alb-kal})
and all $L_p[0,1]$ spaces with $2<p<\infty$, by a classical result of Kadec and Pe{\l}czy\'nski (see \cite{kad-pel62}, Theorem 2 and Corollaries 1 and~4).

\begin{theo}\label{theo:repetitive}
Let $X$ and $Y$ be Banach spaces. Suppose that every seminormalized weakly null sequence in~$X$ 
admits an unconditional basic subsequence whose closed linear span is complemented in~$X$.
Let $(x_n)_{n\in \N}$ be a weakly null sequence in~$X$ and let $(y_n)_{n\in \N}$ be a bounded sequence in~$Y$. 
If the sequence $(x_n\otimes y_n)_{n\in \N}$ is not weakly null in $X\pten Y$, then it admits an $\ell_1$-subsequence.
\end{theo}
\begin{proof}
We can assume that $(x_n\otimes y_n)_{n\in \N}$ is an $\ell_1^+$-sequence. Observe that $(x_n)_{n\in \N}$ cannot be norm null 
and so it admits a seminormalized subsequence, 
say $(x_{n_k})_{k\in \N}$. By the assumption on~$X$, we can assume further that $(x_{n_k})_{k\in \N}$ is an unconditional basic sequence
and that $X_0:=\overline{{\rm span}}(\{x_{n_k}:k\in \N\})$ is complemented in~$X$. Hence, 
the operator $\iota_{X_0}\otimes {\rm id}_Y: X_0 \pten Y \to X\pten Y$ is an isomorphism onto a (complemented) subspace of~$X\pten Y$, where 
$\iota_{X_0}:X_0\to X$ is the inclusion operator and ${\rm id}_Y$ is the identity operator on~$Y$
(see, e.g., \cite[Proposition~2.4]{rya}). Now, since $(x_{n_k} \otimes y_{n_k})_{k\in \N}$ is also an $\ell_1^+$-sequence in $X_0\pten Y$, we can apply
Lemma~\ref{lem:lplus} to conclude that $(x_{n_k} \otimes y_{n_k})_{k\in \N}$ is an $\ell_1$-sequence in $X_0\pten Y$, and so
in $X\pten Y$.
\end{proof}

\subsection{Embedding $\ell_1$ into projective tensor products}\label{subsection:l1-embedding}

Let $X$ and $Y$ be Banach spaces. The subspace of $\mathcal{L}(X,Y^*)$ (resp., $\mathcal{L}(Y,X^*)$) consisting
of all compact operators from $X$ to~$Y^*$ (resp., from $Y$ to~$X^*$) will be denoted by $\mathcal{K}(X,Y^*)$ (resp., $\mathcal{K}(Y,X^*)$).
It is well-known (and not difficult to check) that $\mathcal{L}(X,Y^*)=\mathcal{K}(X,Y^*)$
if and only if $\mathcal{L}(Y,X^*)=\mathcal{K}(Y,X^*)$. The reflexivity of $X\pten Y$ is closely related to those equalities. Indeed, 
if both $X$ and $Y$ are reflexive and $\mathcal{L}(X,Y^*)=\mathcal{K}(X,Y^*)$, then $X\pten Y$ is reflexive; conversely, if $X\pten Y$ is reflexive and, in addition, either $X$ or $Y$ 
has the approximation property, then $\mathcal{L}(X,Y^*)=\mathcal{K}(X,Y^*)$ (see, e.g., \cite[Section~4.2]{rya}).
It is an open problem whether the last statement holds without the approximation property assumption. 

As we already mentioned in the introduction, in \cite{emm2} and \cite{bu-alt} one can find similar results where reflexivity is weakened to ``not containing isomorphic copies of~$\ell_1$''. The purpose of this subsection is to provide more direct proofs of those results. 

\begin{theo}\label{theo:Emmanuele}
Let $X$ and $Y$ be Banach spaces such that one of the following conditions holds:
\begin{enumerate}
\item[(i)] either $X$ or $Y$ has the Dunford-Pettis property;
\item[(ii)] $\mathcal{L}(X,Y^*)=\mathcal{K}(X,Y^*)$.
\end{enumerate}
Then:
\begin{enumerate}
\item[(a)] If $(x_n)_{n\in \N}$ and $(y_n)_{n\in \N}$ are weakly Cauchy sequences in~$X$ and~$Y$, respectively, then $(x_n\otimes y_n)_{n\in \N}$
is weakly Cauchy in~$X\pten Y$.
\item[(b)] If $W_1 \sub X$ and $W_2 \sub Y$ are weakly precompact sets, then $W_1\otimes W_2$
is weakly precompact in~$X\pten Y$.
\item[(c)] If $X$ and $Y$ contain no subspace isomorphic to~$\ell_1$, then $X\pten Y$ contains no subspace isomorphic to~$\ell_1$.
\end{enumerate}
\end{theo}
\begin{proof}
(a) Fix $T\in \mathcal{L}(X,Y^*)$. Note that for every $n,m\in \N$ we have
\begin{multline}\label{eqn:Cauchy}
	\big| \langle T,x_n\otimes y_n\rangle-\langle T,x_m \otimes y_m\rangle \big|  \\ =
	\big| \langle T(x_n),y_n\rangle-\langle T(x_m),y_m\rangle \big|  \leq 
	\big| \langle T(x_n-x_m),y_n\rangle\big| + \big|\langle T(x_m),y_m-y_n\rangle \big|.
\end{multline}
Suppose that $T$ is compact. Then $T$ is completely continuous, hence we have $\|T(x_n-x_m)\|\to 0$ and so $|\langle T(x_n-x_m),y_n\rangle|\to 0$ as $n,m\to\infty$. In addition, since 
the set $\{T(x_m):m\in \N\} \sub Y^*$ is relatively norm compact and $y_m-y_n\to 0$ weakly in~$Y$ as $n,m\to\infty$, we have
$|\langle T(x_m),y_m-y_n\rangle|\to 0$ as $n,m\to\infty$. From \eqref{eqn:Cauchy} it follows that $|\langle T,x_n\otimes y_n\rangle-\langle T,x_m\otimes y_m\rangle|\to 0$ as $n,m\to\infty$. This proves that  $(x_n\otimes y_n)_{n\in \N}$
is weakly Cauchy in~$X\pten Y$ when $\mathcal{L}(X,Y^*)=\mathcal{K}(X,Y^*)$.

If $Y$ has the Dunford-Pettis property, then we have $|\langle T(x_n-x_m),y_n\rangle|\to 0$ as $n,m\to\infty$
(because $T(x_n-x_m)\to 0$ weakly in~$Y^*$ as $n,m\to \infty$ and $(y_n)_{n\in \N}$ is weakly Cauchy)
and $|\langle T(x_m),y_m-y_n\rangle|\to 0$ as $n,m\to\infty$ (because $(T(x_m))_{m\in \N}$ is weakly Cauchy and
$y_m - y_n \to 0$ weakly in~$Y$ as $n,m\to\infty$). Therefore, from~\eqref{eqn:Cauchy} we get
 $|\langle T,x_n\otimes y_n\rangle-\langle T,x_m\otimes y_m\rangle|\to 0$ as $n,m\to\infty$.
This proves that  $(x_n\otimes y_n)_{n\in \N}$ is weakly Cauchy in~$X\pten Y$ when $Y$ has the Dunford-Pettis property.
By symmetry, the same holds whenever $X$ has the Dunford-Pettis property.

(b) is immediate from~(a).

(c) Note that $B_{X\pten Y}=\overline{{\rm conv}}(B_X \otimes B_Y)$ (see, e.g., \cite[Proposition~2.2]{rya})
and that the closed convex hull of a weakly precompact set in a Banach space is weakly precompact as well, according to a result of Stegall
(see \cite[Addendum]{ros-J-7}). The conclusion now follows from~(b) and the fact that a Banach space contains no subspace
isomorphic to~$\ell_1$ if and only if its closed unit ball is weakly precompact.
\end{proof}

The following result is implicit in the proof of \cite[Theorem~6]{kal74}. Recall that an {\em unconditional expansion of the identity} of a Banach space~$X$ is a sequence
$(P_n)_{n\in \N}$ in~$\mathcal{L}(X)$ such that for each $x\in X$ 
we have $x=\sum_{n\in \N}P_n(x)$, the series being unconditionally convergent in~$X$. 

\begin{pro}\label{pro:BUK}
Let $X$ and $Y$ be Banach spaces and let $(P_n)_{n\in \N}$ be an unconditional expansion of the identity of~$X$ (resp., $Y^*$). 
If $X\pten Y$ contains no complemented subspace isomorphic to~$\ell_1$,
then for each $T\in \mathcal{L}(X,Y^*)$ we have $T=\sum_{n\in \N} T\circ P_n$ (resp., $T=\sum_{n \in \N}P_n \circ T$), 
the series being unconditionally convergent in $\mathcal{L}(X,Y^*)$.
\end{pro}
\begin{proof}
Bearing in mind the identification of $(X\pten Y)^*$ and $\mathcal{L}(X,Y^*)$, together with the fact
that a Banach space contains no complemented subspace isomorphic to~$\ell_1$ if and only if
its dual contains no subspace isomorphic to~$c_0$ (see, e.g., \cite[Theorem~4.68]{ali-bur}), 
the assumption turns out to be equivalent to the fact that $\mathcal{L}(X,Y^*)$ contains no subspace isomorphic to~$c_0$.

Suppose that $(P_n)_{n\in \N}$ is an unconditional expansion of the identity of~$X$
(the other case is analogous) and fix $T\in \mathcal{L}(X,Y^*)$. Define $T_J:=\sum_{n\in J} T\circ P_n$ for every finite set $J \sub \N$. For each $x\in X$ we have 
$\sup\{\|T_J(x)\|_{Y^*}: \, J \sub \N \text{ finite}\}<\infty$,
because the series $\sum_{n\in \N} T(P_n(x))$ is unconditionally convergent in~$Y^*$ (with sum $T(x)$).
By the uniform boundedness principle, 
$$
	\sup\{\|T_J\|_{\mathcal{L}(X,Y^*)}:J \sub \N \text{ finite}\}<\infty. 
$$
This implies that $\sum_{n\in \N} T\circ P_n$
is a weakly unconditionally Cauchy series in~$\mathcal{L}(X,Y^*)$, that is, 
for every $\varphi \in \mathcal{L}(X,Y^*)^*$ we have $\sum_{n\in \N}|\langle \varphi,T\circ P_n\rangle|<\infty$.
Since $\mathcal{L}(X,Y^*)$ contains no subspace isomorphic to~$c_0$, 
we conclude that the series $\sum_{n\in \N} T\circ P_n$ is unconditionally convergent in~$\mathcal{L}(X,Y^*)$
(see, e.g., \cite[Theorem~4.49]{ali-bur}). Clearly, its sum equals~$T$.
\end{proof}

Clearly, if a Banach space~$X$ admits an unconditional basis or just 
an unconditional FDD (i.e., unconditional finite-dimensional decomposition), then there is 
an {\em unconditional finite-dimensional expansion of the identity} of~$X$, that is, an unconditional expansion of the identity
consisting of finite rank operators. Of course, this implies that $X$ has the approximation property.
As an immediate consequence of Proposition~\ref{pro:BUK}, we have:

\begin{theo}\label{theo:Bu}
Let $X$ and $Y$ be Banach spaces such that either $X$ or $Y^*$ admits an unconditional finite-dimensional expansion of the identity. 
If $X\pten Y$ contains no complemented subspace isomorphic to~$\ell_1$,
then $\mathcal{L}(X,Y^*)=\mathcal{K}(X,Y^*)$.
\end{theo}

The same argument yields the following:

\begin{rem}\label{rem:ideal}
Let $X$ and $Y$ be Banach spaces such that either $X$ or $Y^*$ admits an unconditional expansion of the identity consisting of elements of 
some norm closed operator ideal~$\mathcal{A}$. If $X\pten Y$ contains no complemented subspace isomorphic to~$\ell_1$,
then all elements of $\mathcal{L}(X,Y^*)$ belong to~$\mathcal{A}$.
\end{rem}

We finish the paper with a question which is open to the best of our knowledge.

\begin{question}\label{question:l1}
Let $X$ and $Y$ be Banach spaces such that $X\pten Y$ contains no complemented subspace isomorphic to~$\ell_1$.
Does the equality $\mathcal{L}(X,Y^*)=\mathcal{K}(X,Y^*)$ hold? What if, in addition, either $X$ or $Y^*$ has the approximation property?
\end{question}

\subsection*{Acknowledgements}

The authors thank A. Avil\'es and G. Mart\'inez-Cervantes for fruitful conversations on the topic of the paper.

The research was supported by grants PID2021-122126NB-C32 (J. Rodr\'{i}guez) and 
PID2021-122126NB-C31 (A. Rueda Zoca), funded by \\ MCIN/AEI/10.13039/501100011033 and ``ERDF A way of making Europe'', and also  
by grant 21955/PI/22 (funded by Fundaci\'on S\'eneca - ACyT Regi\'{o}n de Murcia).
The research of A. Rueda Zoca was also supported by grants FQM-0185 and PY20\_00255 (funded by Junta de Andaluc\'ia).


\bibliographystyle{amsplain}

\end{document}